\documentclass[11pt,leqno]{amsart}
\usepackage{amsmath,amsfonts,latexsym,graphicx,amssymb,url, color}
\usepackage[mathscr]{eucal}
\usepackage{txfonts} 

\usepackage[hmargin=2.5 cm,vmargin=2.0 cm]{geometry}

\usepackage{color} 

\newcommand{\R}{{\mathbb R}} 

\newcommand{\N}{{\mathbb N}}

\newcommand{\e}{\varepsilon}

\newcommand\norm[1]{\left\| #1\right\|}

\newcommand{\M}{{\mathcal M}}

\newcommand{\wei}[1]{\langle #1 \rangle}

\newtheorem{theorem}{Theorem}[section]
\newtheorem{definition}[theorem]{Definition}

\newtheorem{lemma}[theorem]{Lemma}

\numberwithin{equation}{section}

\newcommand{\beq}{\begin{equation}}
\newcommand{\eeq}{\end{equation}}

\definecolor{darkred}{rgb}{.70,.12,.20}

\definecolor{darkgreen}{rgb}{.20,.52,.14}

\title[Local gradient estimates, equations with divergence-free drifts] {A Note on local $W^{1,p}$-regularity estimates for weak solutions of parabolic equations with singular divergence-free drifts}
\author{Tuoc Phan}
\address{Department of Mathematics, University of Tennessee, Knoxville, 227 Ayress Hall, 1403 Circle Drive, Knoxville, TN 37996, U.S.A. }
\email{phan@math.utk.edu}

\begin{document}

\maketitle
\begin{abstract} We investigate weighted Sobolev regularity of weak solutions of non-homogeneous parabolic equations with singular divergence-free drifts.  Assuming that the drifts satisfy  some mild regularity conditions, we establish local weighted $L^p$-estimates for the gradients of  weak solutions. Our results improve the classical one to the borderline case by replacing the $L^\infty$-assumption on solutions by solutions in the John-Nirenberg $\textup{BMO}$ space.
The results are also generalized to parabolic equations in divergence form with small oscillation elliptic symmetric coefficients and therefore improve many known results.
\end{abstract}

\section{Introduction and main results}
We investigate local weighted $L^p$-estimates for the gradients of weak solutions of parabolic equations with low regularity of the divergence-free drifts. A typical example is the parabolic equation
\begin{equation} \label{drift-eqn}
u_t - \Delta u -b\cdot \nabla u = 0, \quad \mathbb{R}^n \times (0, \infty), 
\end{equation}
where the drift $b:\mathbb{R}^n \times (0, \infty) \rightarrow \mathbb{R}^n$ is of divergence-free, i.e. $\text{div}(b(\cdot, t)) =0$ in distributional sense for a.e. $t$.  Due to its relevance in many applications such as in fluid dynamics, and biology, the equation \eqref{drift-eqn} has been investigated by many mathematicians (for example \cite{KS, LS, Sverak, Q-Zhang}). Local boundedness, Harnack's inequality, and H\"{o}lder's regularity are established in \cite{KS, O, Sverak, St,  Q-Zhang} with possible singular drifts. Many other classical results with regular drifts can be found in \cite{Krylov, La, Lib, Softova}. H\"{o}lder's regularity for the fractional Laplace type equations of the form \eqref{drift-eqn} are extensively studied recently (see \cite{CV, KN, Sil}).

Unlike the mentioned work, this note investigates the Sobolev regularity of weak solutions of \eqref{drift-eqn} in weighted spaces. Our goal is to establish local weighted estimates of Calder\'{o}n-Zygmund type for weak solutions of \eqref{drift-eqn} with some mild requirements on the regularity of the drifts $b$.  We study the following parabolic equation that is more general than \eqref{drift-eqn}:
\begin{equation} \label{g-div-eqn}
u_t - \text{div}[a(x,t)\nabla u] - b\cdot \nabla u = \text{div} (F), 
\end{equation}
where  $a = (a^{ij})_{i,j=1}^n$ is a given symmetric $n\times n$ matrix of bounded measurable functions, and $F, b$ are given vector fields with $\text{div}(b) =0$ in distribution sense. The exact required regularity conditions of $a, b, F$ will be specified.

To state our results, we introduce some notation. For each $r >0$, and $z_0 =(x_0, t_0) \in \mathbb{R}^n \times \mathbb{R}$, we denote $Q(z_0)$ the parabolic cylinder in $\R^{n+1}$
\[
Q_r(z_)) = B_r (x_0) \times  \Gamma_r(t_0), \quad \text{where} \quad 
\Gamma_r(t_0) = (t_0 - r^2, t_0 + r^2), \quad \text{and}  \quad B_r(x_0) =\{x \in \mathbb{R}^n:  \|x-x_0\| < r\}.
\]
When $z_0 = (0,0)$, we also write 
\[ Q_r = Q_r(0,0), \quad \text{for} \quad 0 < r <\infty. \] 
As we are interested in the local regularity, we reduce our study to the equation
\begin{equation} \label{main-eqn}
u_t  - \text{div} [a(x,t) \nabla u ] - b\cdot \nabla u = \text{div}(F),  \quad  \text{in} \quad Q_2,
\end{equation}
for given
\[
a: Q_2: \rightarrow \mathbb{R}^{n\times n}, \quad 
b,  F: Q_2 \rightarrow \mathbb{R}^n,
\]
and 
\begin{equation} \label{b-div-free}
\text{div}(b(\cdot, t)) =0, \quad \text{in distribution sense in} \ B_2, \quad \text{for a.e. } \ t \in \Gamma_2.
\end{equation}
For the coefficient matrix $a$, we assume that
\begin{equation} \label{ellipticity}
\left\{
\begin{split}
& a = (a^{ki})_{k,i =1}^n: Q_2 \rightarrow \mathbb{R}^{n\times n} \ \text{is symmetric, measurable}\\
& \qquad \text{ and there exists} \ \Lambda \ \text{such that} \\
& \Lambda^{-1}|\xi|^2 \leq \wei{a(x,t) \xi, \xi} \leq \Lambda |\xi|^2, \quad \quad \text{for a.e. } \quad (x,t) \in Q_2, \quad \text{and} \quad \forall \xi \in \mathbb{R}^n. 
\end{split} \right.
\end{equation}
We also require that the matrix $a$ has a small oscillation. Therefore, we need the following definition.
\begin{definition}  \label{a-BMO} Let $a: Q_2 \rightarrow \mathbb{R}^{n\times n}$ be a measurable matrix valued function. For given $R >0$, we define
\begin{equation*}
[a]_{\textup{BMO}(Q_1)}=\sup_{0<\rho\leq 1}\sup_{(y,s)\in \overline{Q_1}} \frac{1}{|Q_\rho(y, s)|}\int_{Q_\rho(y, s)}{|a(x, t) - \bar{a}_{B_\rho(y)}(t)|^2\, dx dt},
\end{equation*}
where $\bar{a}_{U}(t) = \fint_{U}{a(x,t )\, dx}$ is the average of $a$ in the set $U \subset B_2$.\end{definition}
\noindent
For the regularity of the vector field $b$, we need the following function space, which was introduced in \cite{MV-1, PP}
\begin{definition} For $x_0 \in \mathbb{R}^{n}$ and $r>0$, a locally square integrable function $f$ defined in a neighborhood of $B_r(x_0)$ is said to be in $\mathcal{V}^{1,2}(B_r(x_0))$  if there is $k \in [0, \infty)$ such that
\begin{equation} \label{X}
\int_{B_{r}(x_0)} |f(x)|^2 \varphi (x)^2 dx \leq k \int_{B_{r}(x_0)} |\nabla \varphi(x)|^2 dx, \quad \forall \varphi \in C_0^\infty(B_r(x_0)).
\end{equation}
 We denote
\[
\norm{b}_{\mathcal{V}^{1,2}(B_r(x_0))}^2 = \inf\left\{ k \in [0, \infty) \ \text{such that} \ \eqref{X}\ \text{holds} \right \}.
\]
\end{definition}
\noindent In this paper, the numbers $s, s' \in (1, \infty), \alpha >0$ and $\lambda$ are fixed and satisfying
\begin{equation} \label{s-s'-alpha}
\frac{1}{s'}+ \frac{1}{s} =1, \quad -(n+2) \leq \lambda \leq s', \quad \alpha = \lambda(s-1).
\end{equation}
We also denote $L^p(Q, \omega)$ the weighted Lebesgue space with weight $\omega$:
\[
L^p(Q, \omega) = \Big\{f : Q \rightarrow \R : \|f\|_{L^p(Q, \omega)} : = \left( \int_Q |f(x,t)|^p \omega(x,t) dxdt \right)^{1/p} < \infty \Big\}, \quad 1 < p < \infty.
\]
At this moment, we refer the readers to Section \ref{def-section} for the definition of weak solutions of \eqref{main-eqn}, the definition of of Muckenhoupt $A_{q}$ weights, and the definition of fractional Hardy-Littlewood maximal functions $\M_\alpha$. 
Our main result is the following theorem on local weighted $W^{1,p}$-regularity estimates for weak solutions of \eqref{main-eqn}.
\begin{theorem} \label{g-theorem} Let $\Lambda, M_0$ be positive numbers, $p \in (2,\infty)$, and $\omega \in A_{p/2}$. Let $s, s', \lambda, \alpha$ be as in \eqref{s-s'-alpha}. Then, there exists a sufficiently small number $\delta = \delta(\Lambda, M_0, s, \lambda, [\omega]_{A_{p/2}}, p, n) >0$ such that the following holds: Suppose that $a$ satisfies \eqref{ellipticity}, $F \in L^2(Q_2)$, and $b \in L^\infty (\Gamma_2, \mathcal{V}^{1,2}(B_2))$ such that \eqref{b-div-free} holds,  and
\[
[[a]]_{\textup{BMO}(Q_1)} < \delta,  \quad \norm{b}_{L^\infty (\Gamma_2, \mathcal{V}^{1,2}(B_2))} \leq M_0.
\]
Then for every weak solution $u$ of \eqref{main-eqn},  the following estimate holds
\begin{equation} \label{main-est}
\begin{split}
\int_{Q_1} |\nabla u|^p\omega(z) dz  &\leq  C \left[ \norm{F}_{L^p(Q_2,\omega)}^p +  [u]_{s', \lambda, Q_1}^p\norm{\M_{\alpha, Q_2}(|b|^s)^{2/s}}_{L^{p/2}(Q_2, \omega)}^{p/2} + \omega(Q_1) \norm{\nabla u}_{L^2(Q_2)} ^p \right],
 \end{split}
\end{equation}
as long as its right hand side is finite. Here,  $[u]_{s', \lambda, Q_1}$ is the parabolic semi-Campanato's norm of $u$ on $Q_1$,
\[
[u]_{s', \lambda, Q_1} = \sup_{0 <\rho< 1, z \in Q_1} \left[\rho^{-\lambda} \fint_{Q_\rho(z)} |u(x,t) - \bar{u}_{Q_\rho(z)}|^{s'} dxdt\right]^{1/s'},
\]
and $C>0$ is a constant depending only on $\Lambda, M_0, s, \lambda, p, n$ and $[\omega]_{A_{p/2}}$.
\end{theorem}
\noindent We now point out  a few remarks regarding Theorem \ref{g-theorem}. Firstly, observe that the standard Calder\'{o}n-Zygmund theory can be applied directly to \eqref{main-eqn} to obtain
\[
\norm{\nabla u}_{L^p(Q_1)} \leq C\left[ \norm{u}_{L^\infty(Q_2)} \norm{b}_{L^p(Q_2)} + \cdots \right],
\]
as long as $u \in L^\infty(Q_2)$. Theorem \ref{g-theorem} improves this Calder\'{o}n-Zygmund estimate theory for the equation \eqref{main-eqn} to the borderline case, replacing the assumption $u \in L^\infty(Q_2)$ by $u \in \text{BMO}(Q_1)$. Indeed, if we take $\lambda =0$ (and then $\alpha =0$), then the estimate \eqref{main-est} reduces to
\begin{equation} \label{borderline-case}
\norm{\nabla u}_{L^p(Q_1, \omega)} \leq C\left[ \norm{u}_{\text{BMO}(Q_1)} \norm{b}_{L^p(Q_2, \omega)} + \cdots \right].
\end{equation}
Secondly, the weighted $W^{1,p}$-regularity estimates are useful in some applications. For example, in \cite{BDS, BS}, the weighted $W^{1,p}$-regularity estimates are key ingredients for proving the existence and uniqueness of very weak solutions of some classes of elliptic equations. Moreover, with some specific choice of $\omega$, the weighted estimate \eqref{main-est} is known to produce the regularity estimates for $\nabla u$ in Morrey spaces, see for example \cite{Adams, Byun-P, Fazio, MP-1}. Lastly, when $\alpha >0$, because $\M_\alpha \leq \mathbf{I}_\alpha$, the Riesz potential of order $\alpha$, we observe that the fractional Hardy-Littlewood maximal function of order $\alpha$ of $b$, i.e. $\M_\alpha (|b|^s)^{2/s}$, is more regular than $b$.  This fact enables the estimate \eqref{main-est} to be useful in some applications. To see this, we just simply consider the stationary case (i.e. $u$ is time independent), $n \geq 3$ and $s =2$. Assume, for example, that $b \in L^{n,\infty}(B_2) \subset  \mathcal{V}^{1,2}(B_2)$, where $L^{n,\infty}$ is the weak $L^n$-space, and assume also that $F$ is regular enough. Then, it is proved in \cite{Sverak, Q-Zhang} that $u$ is H\"{o}lder. Therefore, $[u]_{2, \lambda, B_1} <\infty$ with some $\lambda >0$. From this, and \eqref{s-s'-alpha}, we see that $\alpha >0$, and we then can find some small constant $\e_0 >0$ such that
\[
\norm{\M_{\alpha, B_2}(|b|)}_{L^p(B_2)} \leq C\norm{b}_{L^{n,\infty}(B_2)} < \infty, \quad \text{for all} \quad p < n+\e_0.
\]
Therefore,  \eqref{main-est} gives the estimate of $\|\nabla u\|_{L^{p}(B_1)}$ with some $p \in [2, n+\e_0)$.  This estimate with $p>n$ is useful in \cite{Kim-Tsai} to prove the regularity, and uniqueness of very weak $W^{1,q}$-solution of the stationary equation of \eqref{main-eqn}, with $1 < q<2$. Details of this discussion and its application can be also found in \cite{Tsai-Phan}. 

We finally would like to point out that the space $\mathcal{V}^{1,2}(\mathbb{R}^n)$ is already appeared in \cite{MV-1, PP, Q-Zhang}. In particular, in \cite{Q-Zhang}, the space $L^\infty_t(\mathcal{V}^{1,2}(\mathbb{R}^n))$ is used to study the boundedness of weak solution of the equation  \eqref{drift-eqn}.  For $n\geq 3$, the space $\mathcal{V}^{1,2}(\mathbb{R}^n)$ is already appeared in \cite{MV-1, PP}. Moreover, it is known that (see \cite{PP}) 
\begin{equation} \label{morrey-embed}
L^n(\mathbb{R}^n) \subset \mathscr{M}^{p,p}(\mathbb{R}^n) \subset \mathcal{V}^{1,2}(\mathbb{R}^n)   \quad \forall \  2 < p \leq n, 
\end{equation}
and therefore 
\[L^\infty_t(L^n(\mathbb{R}^n)) \subset L^\infty_t (\mathscr{M}^{p,p}(\mathbb{R}^n)) \subset L^\infty_t(\mathcal{V}^{1,2}(\mathbb{R}^n)), \] 
where $\mathscr{M}^{p,p}(\mathbb{R}^n)$ denotes the homogeneous Morrey space. Specifically, for $0 < p \leq n$ and $0~<~\lambda~<~p$, the function $f \in L^p_{\text{loc}}(\mathbb{R}^n)$ belongs to the space $\mathscr{M}^{p,\lambda}(\mathbb{R}^n)$ if 
\[
\norm{f}_{\mathscr{M}^{p, \lambda}(\mathbb{R}^n)} = \sup_{B_r(x_0) \subset \mathbb{R}^n} \left\{r^{\lambda -n} \int_{B_r(x_0)} |f(x)|^p
 \right\}^{\frac{1}{p}} <\infty.
\]

We use perturbation approach introduced in \cite{CP} to prove Theorem \ref{g-theorem}. Our approach is also influenced by \cite{BW2, LTT, MP-1, NP, Wang}. To implement the approach, we introduce the function $B(x,t) = \big([u]_{s', \lambda, Q_1} |b(x,t)|\big)^s$, which is invariant under the standard dilation, and translation. This function also captures the cancellation due to the divergence-free of the vector field $b$, which is the main reason so that the estimate \eqref{borderline-case} holds  the borderline case. The results on the doubling property and reverse H\"{o}lder's inequality for the Muckenhoupt weights  due to R. R. Coifman, and C. Fefferman in \cite{Coif-Feffer} are also used frequently to derive the weighted estimates.

We conclude the section by introducing the organization of the paper. Section \ref{def-section} gives definitions, notations, and some preliminaries results needed in the paper. Some simple energy estimates for weak solutions of \eqref{main-eqn} is given in Section \ref{energy-section}. The main step in the perturbation technique, the approximation estimates, is carried out in Section \ref{Approximation-Sec}. Section \ref{density-est-section} is about the proof of Theorem \ref{g-theorem}.

\section{Definitions of weak solutions, and preliminaries on weighted inequalities} \label{def-section}
\subsection{Definitions of weak solutions}
For each $z_0 = (x_0, t_0) \in \mathbb{R}^n \times \mathbb{R}$, and for any parabolic cylinder $Q_R(z_0)$, we denote $\partial_p Q_R(z_0)$ the parabolic boundary of $Q_R(z_0)$, i.e.
\[
\partial_p Q_R(z_0) = ( B_R(x_0) \times \{t_0 -R^2 \} ) \cup  (\partial B_R (x_0) \times [t_0-R^2, t_0 +R^2]).
\]
The following standard definitions of  weak solutions are also recalled.
\begin{definition} \label{weak-sol-Q} Let $Q_r $ be a parabolic cube. For every $f \in \textup{L}^2(Q_r), F, b \in \textup{L}^2(Q_r)^n$, we say that $u$ is a weak solution of
\[
u_t -\textup{div}[a \nabla u] - b\cdot \nabla u = \textup{div}(F) + f, \quad \text{in} \quad Q_r,
\]
if $u \in \textup{L}^2(\Gamma_{r}, H^1(B_{r}))$, $u_t \in \textup{L}^2(\Gamma_r, \textup{H}^{-1}(B_r))$, and   
\[
\int_{\Gamma_r}\wei{u_t, \varphi}_{\textup{H}^{-1}(B_r) , \textup{H}^1_0(B_r))} dt + \int_{Q_r} \Big[ \wei{a \nabla u, \nabla \varphi} - b\cdot \nabla u \varphi \Big] dx dt= \int_{Q_r}[f \varphi - \wei{F, \nabla \varphi}] dx dt,
\]
for all $ \varphi \in \{\phi \in C^\infty(\overline{Q}_r) : \phi =0 \ \text{on} \ \partial_p Q_r\}$.
\end{definition}
\noindent The following definition of weak solution is also needed in the paper.
\begin{definition} \label{weak-sol-BC} Let $Q_r $ be a parabolic cube. For every $f \in \textup{L}^2(Q_r), F, b \in \textup{L}^2(Q_r)^n$, and for $g \in L^2(\Gamma_r, H^1(B_r))$, we say that $u$ is a weak solution of
\[ \left\{
\begin{array}{cccl}
 u_t -\textup{div}[a\nabla u]  - b\cdot \nabla u & = & \textup{div}(F) + f,  &  \quad\text{in} \quad Q_r, \\
u & = & g, &  \quad  \text{on} \quad \partial_p Q_r,
\end{array} \right.
\]
if $u$ is a weak solution of
\[
u_t -\textup{div}[a \nabla u] -b \cdot \nabla u= \textup{div}(F) + f, \quad \text{in} \quad Q_r, \\
\]
in the sense of \textup{Definition \ref{weak-sol-Q}} and $u - g \in \{\phi \in  \textup{L}^2(\Gamma_r,  \textup{H}^1(B_r) : \phi =0 \ \text{on} \ \partial_p Q_r\}$.
\end{definition}
\subsection{Munckenhoupt weights and Hardy-Littlewood maximal functions}

For each $1 \leq q < \infty$, a non-negative, locally integrable function $\mu :\R^{n+1} \rightarrow [0, \infty)$ is said to be in the class of parabolic $A_q$ of Muckenhoupt weights if 
\[
\begin{split}
[\mu]_{A_q} & :=  \sup_{r >0, z \in \R^{n+1}} \left(\fint_{Q_r(z)} \mu (x,t) dx dt  \right) \left(\fint_{Q_r(z)} \mu (x,t)^{\frac{1}{1-q}} dx dt \right)^{q-1} < \infty, \quad \textup{if} \quad q > 1, \\
[\mu]_{A_1} &: =  \sup_{ r>0 , z \in \R^{n+1}} \left(\fint_{Q_r(z)} \mu (x, t) dx dt \right)  \norm{\mu^{-1}}_{L^\infty(Q_{r}(z))}  < \infty \quad \textup{if} \quad q  =1.
\end{split}
\]
It is well known that the class of $A_p$-weights satisfies the reverse H\"{o}lder's inequality and the doubling properties, see for example \cite{Coif-Feffer, Fabes-Riviere, Stein}. In particular, a measure with an $A_p$-weight density is, in some sense, comparable with the Lebesgue measure. 
\begin{lemma}[\cite{Coif-Feffer}] \label{doubling} For $1 < q < \infty$, the following statements hold true
\begin{itemize}
\item[\textup{(i)}] If $\mu \in A_{q}$,  then for every parabolic cube $Q \subset \R^{n+1}$ and every measurable set $E\subset Q$, 
$
\mu(Q) \leq [\mu]_{A_{q}} \left(\frac{|Q|}{|E|}\right)^{p} \mu(E).
$
\item[\textup{(ii)}] If $\mu \in A_q$, then there is $C = C([\mu]_{A_q}, n)$ and $\beta = \beta([\mu]_{A_q}, n) >0$ such that
$
\mu(E) \leq C \left(\frac{|E|}{|Q|} \right)^\beta \mu(Q),
$
 for every parabolic cube $Q \subset \R^{n+1}$ and every measurable set $E\subset Q$.
\end{itemize}
\end{lemma} \noindent
Let us also recall the definition of the parabolic fractional Hardy-Littlewood maximal operators which will be needed in the paper
\begin{definition} \label{Maximal-Operator} Let $\alpha \in \mathbb{R}$, the parabolic Hardy-Littlewood fractional maximal function of order $\alpha$ of a locally integrable function  $f$ on $\R^n$ is defined by
\[
 (\M_\alpha f)(x,t) = \sup_{\rho>0}\rho^\alpha \fint_{Q_\rho(x,t)}{|f(y,s)|\, dy ds}.
\]
If $f$ is defined in a region $U\subset \R^n\times \R$, then we   denote
\[
 \M_{\alpha, U} f = \M_\alpha (\chi_U f).
\]
Moreover, when $\alpha =0$, we write
\[
\M f = \M_0 f, \quad \M_U f = \M_{0, U} f.
\]
\end{definition}
The following boundedness of the Hardy-Littlewood maximal operator is due to Muckenhout \cite{Muckenhoupt}. For the proof of this lemma, one can find it in \cite{Fabes-Riviere, Stein}.
\begin{lemma} \label{Hardy-Max-p-p} Assume that $\mu \in A_q$ for some $1 < q < \infty$. Then, the followings hold.
\begin{itemize}
\item[(i)] Strong $(q,q)$: There exists a constant $C = C([\mu]_{A_q}, n,q)$ such that  
\[ \|\mathcal{M} \|_{L^{q}(\R^{n+1},\mu) \to L^{q}(\R^{n+1}, \mu)} \leq C. \] 
\item[(ii)] Weak $(1,1)$: 
There exists a constant $C=C(n)$ such that for any $\lambda >0$, we have 
\[
\big| \big \{ (x,t) \in \mathbb{R}^{n+1}: \mathcal{M} (f) > \lambda \big\}\big| \leq \frac{C}{\lambda} \int_{\mathbb{R}^{n+1}}|f(x,t)| dxdt.
\]  
\end{itemize}
\end{lemma}

\subsection{Some useful measure theory lemmas}
We collect some results needed in the paper. Our first lemma is the standard result in in measure theory.
\begin{lemma} \label{measuretheory-lp}
Assume that $g\geq 0$ is a measurable function in a bounded subset $U\subset \mathbb{R}^{n+1}$. Let $\theta>0$ and $\varpi>1$ be given constants. If $\mu$ is a weight in  $L^{1}_{loc}(\mathbb{R}^{n+1})$, then for any $1\leq p < \infty$ 
\[
g\in L^{p}(U,\mu) \Leftrightarrow S:= \sum_{j\geq 1} \varpi^{pj}\mu(\{x\in U: g(x)>\theta \varpi^{j}\}) < \infty. 
\]
Moreover, there exists a constant $C>0$ such that 
\[
C^{-1} S \leq \|g\|^{p}_{L^{p}(U,\mu)} \leq C (\mu(U) + S), 
\]
where $C$ depends only on $\theta, \varpi$ and $p$. 
\end{lemma} \noindent
The following lemma is commonly used, and it is a consequence of the Vitali's covering lemma. The proof of this lemma can be found in \cite[Lemma 3.8]{MP-1}. 
\begin{lemma}\label{Vitali} Let $\mu$ be an $A_{q}$ weight for some $q \in (1,\infty)$ be a fixed number. Assume that $E \subset K \subset Q_1$ are measurable sets for which there exists $\epsilon, \rho_{0}\in (0, 1/4)$ such that 
\begin{itemize}
\item[(i)]  $\mu(E) < \epsilon \mu(Q_{1}(z)) $ for all $z\in \overline{Q}_{1}$, and 
\item[(ii)] for all $z \in Q_{1}$ and $\rho \in (0, \rho_{0}]$, if $\mu (E \cap Q_{\rho}(z)) \geq \epsilon \mu(Q_{\rho}(z))$, 
then $Q_{\rho}(z)\cap Q_{1} \subset K. $
\end{itemize}  
Then with $\e_1 = \e  (20)^{nq}[\mu]_{A_q}^2 $ so that the following estimate holds 
\[
\mu(E) \leq \epsilon_1 \,\mu(K). 
\] 
\end{lemma}
\section{Caccioppoli's type estimates} \label{energy-section}
 
Suppose that $a$ satisfies \eqref{ellipticity}, and $b \in L^\infty(\Gamma_2, \mathcal{V}^{1,2}(B_2))^n \cap L^2(Q_2)^n$ with $\text{div}(b) =0$. In this section, let $u$ be a weak solution of
\[
u_t - \text{div} [a(x,t)  \nabla u] -   b(x,t)\cdot \nabla u = \text{div}(F), \quad \text{in} \quad Q_{2}.
\]
Also, let $v$ be a weak solution of
\[
\left\{
\begin{array}{cccl}
v_t - \text{div}[\bar{a}_{B_{7/4}}(t) \nabla v ] &= &0, & \quad Q_{7/4}, \\
v &= & u, &  \quad \partial_p Q_{7/4}.
\end{array}
\right.
\]
The meanings for weak solutions of these equations are given in Definition \ref{weak-sol-Q} and Definition \ref{weak-sol-BC}, respectively. We will derive some fundamental estimates for $u$ and $v$. 

\begin{lemma}  \label{w-energy} Let $w = u - v$, then there exists a constant $C$ depending on only $\Lambda, n$ such that
\[
\begin{split}
& \sup_{t \in \Gamma_{7/4}}\int_{B_{7/4}} w^2(x,t) dx + \int_{Q_{7/4}} |\nabla w|^2 dx dt \\  
& \leq C\left[ \Big( \norm{  b}_{L^\infty(\Gamma_2, \mathcal{V}^{1,2}(B_2))}^2 +1\Big)  \int_{Q_{7/4}} |\nabla u|^2 dxdt + \int_{Q_{7/4}} |F|^2 dx dt \right].
\end{split}
\]
\end{lemma}
\begin{proof} Note that $w$ is a weak solution of
\[
\left\{
\begin{array}{cccl}
w_t - \text{div}[\bar{a}_{B_{7/4}}(t) \nabla w + (a - \bar{a}_{B_{7/4}}(t)) \nabla u] -   b\cdot  \nabla u &= & \text{div} (F), & \quad \text{in} \quad  Q_{7/4}, \\
w & = &0, &  \quad \text{on} \quad \partial_p Q_{7/4}.
\end{array}
\right.
\]
Multiplying the equation with $w$, and using the integration by parts in $x$, we see that
\[
\begin{split}
&\frac{1}{2}\frac{d}{dt} \int_{B_{7/4}} w^2(x,t) dx + \int_{B_{7/4}} \wei{\bar{a}_{B_{7/4}}(t) \nabla w, \nabla w} dx \\
& = - \int_{B_{7/4}} \wei{(a -\bar{a}_{B_{7/4}}(t)) \nabla u, \nabla w}  dx+    \int_{B_{7/4}} [b \cdot \nabla u] w dx  -\int_{B_{7/4}} F \cdot \nabla w dx .
 \end{split}
\]
Then, by integrating this equality in time and using the ellipticity condition \eqref{ellipticity}, we obtain
\begin{equation} \label{w-test-1}
\begin{split}
& \frac{1}{2}\sup_{\Gamma_{7/4}} \int_{B_{7/4}} w^2 dx + \Lambda^{-1} \int_{Q_{7/4}} |\nabla w|^2 dxdt \\
&\leq  \int_{Q_{7/4}}| \wei{(a -\bar{a}_{B_{7/4}}(t)) \nabla u, \nabla w}| dxdt +    \int_{Q_{7/4}} |b \cdot \nabla u w| dxdt + \int_{Q_{7/4}}\Big[|F \cdot \nabla w| \Big] dxdt.
\end{split}
\end{equation}
We now estimate terms by terms of the right hand side of \eqref{w-test-1}. 
From H\"{o}lder's inequality, and the Young's inequality, and the fact that $w=0$ on $\partial_p Q_{7/4}$, the second term in the right hand side of \eqref{w-test-1} can be estimated as
\begin{equation} \label{MV-inq}
\begin{split}
 \int_{Q_{7/4}}|  b||w| |\nabla u| dxdt  & \leq \left\{\int_{Q_{7/4}}|  b|^2 w^2 dxdt \right\}^{1/2} \left\{\int_{Q_{7/4}} |\nabla u|^2 dxdt \right\}^{1/2}\\
 & \leq \norm{  b}_{L^\infty(\Gamma_2, \mathcal{V}^{1,2}(B_2))} \left\{\int_{Q_{7/4}}|\nabla w|^2 dxdt \right\}^{1/2} \left\{\int_{Q_{7/4}} |\nabla u|^2 dxdt \right\}^{1/2}\\
 & \leq \frac{\Lambda^{-1}}{6} \int_{Q_{7/4}}|\nabla w|^2 dxdt + C(\Lambda) \norm{  b}_{L^\infty(\Gamma_2, \mathcal{V}^{1,2}(B_2))}^2 \int_{Q_{7/4}} |\nabla u|^2 dxdt.
 \end{split}
\end{equation}
On the other hand, by the boundedness of $a$ in  \eqref{ellipticity}, and the H\"{o}lder's inequality, we conclude that
\[
\begin{split}
& \int_{Q_{7/4}}| \wei{(a -\bar{a}_{B_{7/4}}) \nabla u, \nabla w}| dx dt  \leq C(\Lambda) \int_{Q_{7/4}} |\nabla u|^2 dxdt  + \frac{\Lambda^{-1}}{6}\int_{Q_{7/4}} |\nabla w|^2 dxdt, \quad \text{and} \\
& \int_{Q_{7/4}} |F \cdot \nabla w| dx dt   \leq C(\Lambda) \int_{Q_{7/4}}|F|^2dxdt + \frac{\Lambda^{-1}}{6} \int_{Q_{7/4}} |\nabla w|^2 dxdt.
 \end{split}
\]
Collecting all of the estimates, we obtain from \eqref{w-test-1} that 
\[
\begin{split}
&\frac{1}{2}\sup_{\Gamma_{7/4}} \int_{B_{7/4}} w^2(x,t) dx  + \Lambda^{-1} \int_{Q_{7/4}} |\nabla w|^2 dx dt \\
&\leq \frac{\Lambda^{-1}}{2} \int_{Q_{7/4}} |\nabla w|^2 dx dt 
+ C \left(\Big[\norm{  b}_{L^\infty(\Gamma_2, \mathcal{V}^{1,2}(B_2))}^2 + 1\Big]\int_{Q_{7/4}} |\nabla u|^2 dx dt +  \int_{Q_{7/4}}  |F|^2 dx \right).
\end{split}
\]
Therefore,
\[
\begin{split}
& \sup_{t \in \Gamma_{7/4}}\int_{B_{7/4}} w^2(x,t) dx +  \int_{Q_{7/4}} |\nabla w|^2 dx dt \\ 
& \leq C (\Lambda) \left(\Big[ \norm{  b}_{L^\infty(\Gamma_2, \mathcal{V}^{1,2}(B_2))}^2 +1\Big] \int_{Q_{7/4}}|\nabla u|^2 dx dt+  \int_{Q_{7/4}} |F|^2 dxdt \right).
\end{split}
\]
The proof is complete.
\end{proof}
\noindent
The following version of local energy estimate for $w = u -v$ is also needed.
\begin{lemma} \label{w-local-energy}  There exists a constant $C = C_0$ depending only on $\Lambda, n$ such that for $w = u-v$, and for every smooth, non-negative cut-off function $\varphi \in C_0^\infty(Q_r)$ with $0 < r \leq {7/4}$, there holds
\[
\begin{split}
& \sup_{t\in \Gamma_r}\int_{B_r} w^2 \varphi^2 dx + \int_{Q_r} |\nabla w|^2\varphi^2 dx dt \\
& \leq C_0 \left\{\left[\norm{  b}_{L^\infty(\Gamma_2, \mathcal{V}^{1,2}(B_2))}^2 +1 \right] \int_{Q_{r}} w^2 \Big[\varphi^2 + |\partial_t \varphi|^2 + |\nabla \varphi|^2] dxdt   +  \int_{Q_r} |F|^2 \varphi^2 dx dt \right. \\
& \qquad + \left.\norm{  \nabla v \varphi}_{L^\infty(Q_{7/4})} \norm{b}_{L^\infty(\Gamma_2, \mathcal{V}(B_2))} \norm{ \nabla \varphi}_{L^2(Q_{7/4})} \left\{ \int_{Q_r}  |w|^2\varphi^2dxdt\right\}^{1/2}
   + \norm{|\nabla v| \varphi}_{L^\infty(Q_r)}^2\int_{Q_r}| a - \bar{a}_{B_{7/4}}(t)|^2 dx dt \right \}.
\end{split}
\]
\end{lemma}
\begin{proof} We write $Q = Q_r, B= B_r$, and $\Gamma = \Gamma_r$. Note that $w$ is a weak solution of
\[
w_t -\text{div}[a \nabla w + (a - \bar{a}_{B_{7/4}}) \nabla v] -   b \cdot \nabla w -   b  \cdot \nabla v = \text{div}(F), \quad \text{in} \quad Q_{7/4}.
\]
By using  $w\varphi^2$ as a test function of the equation of $w$, we obtain
\begin{equation}  \label{w-energy-test}
\begin{split}
& \frac{1}{2} \frac{d}{dt} \int_B w^2(x,t) \varphi^2(x,t) dx + \int_B \wei{a\nabla w, \nabla w}\varphi^2 dx \\
& = -\int_B \wei{a\nabla w, \nabla (\varphi^2)} w dx 
- \int_{B} \wei{(a-\bar{a}_{B_{7/4}}(t)) \nabla v, \varphi^2\nabla w + 2w\varphi \nabla \varphi}dx  \\
& \qquad +  \int_{B} [b\cdot \nabla w] w \varphi^2 dx +   \int_{B}[ b\cdot  \nabla v] w\varphi^2 dx \\
& \qquad - \int_B\wei{F, \nabla (w\varphi^2)} +  \int_B w^2 \varphi \varphi_t dx .
\end{split}
\end{equation}
Note again that the second term in the left hand side of \eqref{w-energy-test} can be estimated using \eqref{ellipticity} as
\[
\int_Q \wei{a\nabla w, \nabla w}\varphi^2 dx dt  \geq \Lambda^{-1} \int_Q |\nabla w|^2 \varphi^2 dxdt.
\]
Also, from the integration by parts in $x$, and $\text{div}(b) =0$, we also have
\[
\begin{split}
  \int_B [b\cdot \nabla w ] w \varphi^2 dx  & = \frac{1}{2} \int_{B} [b\cdot \nabla (w^2)]\varphi^2 dx= -  \int_{B} [b\cdot \nabla \varphi] \varphi w^2 dx.
\end{split}
\]
Hence, \eqref{w-energy-test} implies
\[
\begin{split}
& \frac{1}{2} \frac{d}{dt} \int_B w^2(x,t) \varphi^2(x,t) dx + \Lambda^{-1}\int_B |\nabla w|^2\varphi^2 dx \\
& \leq  \int_B |\wei{a\nabla w, \nabla (\varphi^2)} w| dx 
+ \int_{B} |\wei{(a-\bar{a}_{B_{7/4}}(t)) \nabla v, \varphi^2\nabla w + 2w\varphi \nabla \varphi}| dx  \\
& \quad +   \int_{B} |[b\cdot \nabla \varphi] w^2 \varphi^2| dx +   \int_{B}|[ b\cdot  \nabla v] w\varphi^2| dx \\
& \qquad + \int_B \Big[ |\wei{F, \nabla (w\varphi^2)}| + 2 w^2 |\varphi \varphi_t|\Big] dx .
\end{split}
\]
By integrating this inequality in time, and using the $L^\infty$-bound of $a$ from \eqref{ellipticity}, we infer that  
\begin{equation}  \label{w-energy-test-l}
\begin{split}
& \frac{1}{2}\sup_{t \in \Gamma_r} \int_B w^2(x,t) \varphi^2(x,t) dx + \Lambda^{-1}\int_Q |\nabla w|^2\varphi^2 dxdt \\
& \leq  2  \int_{Q} |\nabla w| |\nabla \varphi| |\varphi w| dx dt 
+  \int_{Q} |a-\bar{a}_{B_{7/4}}| | \nabla v| \Big[ \varphi^2|\nabla w| + 2|w|\varphi|| \nabla \varphi| \Big] dx dt  \\
& \quad +   \int_{Q} |b| |\nabla \varphi| w^2 \varphi^2 dx dt +   \int_{Q}|b|| \nabla v| |w|\varphi^2 dx dt \\
& \qquad + \int_Q \Big[|\wei{F, \nabla (w\varphi^2)}| + 2w^2 |\varphi \varphi_t|\Big] dx dt .
\end{split}
\end{equation}
We now pay particular attention to the terms in the right hand side of \eqref{w-energy-test-l} involving $  b$, as other terms can be estimated exactly as in Lemma \ref{w-energy}. By using the H\"{o}lder's inequality and Young's inequality, we see that
\[
\begin{split}
\int_{Q} w^2  \varphi |  b||\nabla \varphi| dxdt & \leq \left\{  \int_Q  |  b|^2 w^2 \varphi^2\right\}^{1/2} 
\left\{\int_Q  w^2 |\nabla \varphi|^2 dxdt  \right\}^{\frac{1}{2}} \\
& \leq  \norm{  b}_{L^\infty(\Gamma_{7/4}\ \mathcal{V}^{1,2}(B_{7/4}))} \left\{ \int_Q |\nabla ( w \varphi)|^2 dxdt \right\}^{1/2} 
\left\{\int_Q  w^2|\nabla \varphi|^2 dxdt  \right\}^{\frac{1}{2}} \\
& \leq \epsilon  \int_Q |\nabla w|^2\varphi^2 dxdt + 
C(\epsilon) \norm{  b}_{L^\infty(\Gamma_{7/4}\ \mathcal{V}^{1,2}(B_{7/4}))}^{2}\int_Q w^2|\nabla \varphi|^2 dxdt,
 \end{split}
\]
for any arbitrary $\epsilon >0$. Similarly, we also obtain
\[
\begin{split}
  \int_{Q} |b| |\nabla v| |w| \varphi^2 dxdt & \leq\norm{  \nabla v \varphi}_{L^\infty(Q_{7/4})}\left\{\int_{Q}|b|^2 \varphi^2 dxdt\right\}^{1/2} \left\{ \int_{Q}  |w|^2\varphi^2dxdt\right\}^{1/2} \\
&\leq \norm{  \nabla v \varphi}_{L^\infty(Q_{7/4})} \norm{b}_{L^\infty(\Gamma_{7/4}, \mathcal{V}(B_{7/4}))} \norm{ \nabla \varphi}_{L^2(Q_{7/4})} \left\{ \int_{Q}  |w|^2\varphi^2dxdt\right\}^{1/2}.
\end{split}
\]
Other terms can be estimated similarly. Then, collecting all the estimates and choose $\epsilon$ sufficiently small,  we obtain the desired result.
\end{proof}
\section{Approximation estimates} \label{Approximation-Sec}
We apply the "freezing coefficient" technique to establish the regularity estimates for weak solutions of \eqref{main-eqn}. To do this, we approximate the weak solution $u$ of the equation
\begin{equation} \label{u-Q5}
u_t - \text{div}[a\nabla u] -   b\cdot \nabla u = \text{div}(F)  \quad \text{in} \quad Q_2,
\end{equation}
by the weak solution $v$ of the equation
\begin{equation} \label{v-Q4}
\left\{
\begin{array}{cccl}
v_t - \text{div}[ \bar{a}_{B_{7/4}}(t) \nabla v]  &= & 0, & \quad \text{in} \quad Q_{7/4},\\
v & = &u, &  \quad \text{on} \quad \partial_p Q_{7/4}
\end{array} \right.
\end{equation}
Again, the meanings for weak solutions of equations \eqref{u-Q5} - \eqref{v-Q4} are given in Definition \ref{weak-sol-Q} and Definition \ref{weak-sol-BC}, respectively. 
We essentially follow the method in our recent work \cite{LTT, NP}, which in turn is influenced by \cite{BW2, CP, PS, Wang}. We first begin with the standard result on the regularity of weak solution of the constant coefficient equation \eqref{v-Q4}. 
\begin{lemma} \label{reg-v} There exists a constant $C$ depending only on the ellipticity constant $\Lambda$ and $n$ such that if $v$ is a weak solution of 
\[ v_t - \text{div}[\bar{a}_{B_{7/4}}(t) \nabla v] =0 \quad \text{in} \quad Q_{7/4}, \]
 then
\[
\|\nabla v\|_{L^\infty(Q_{\frac{3}{2}})} \leq C \left( \fint_{Q_{7/4}} |\nabla v(x,t)|^2 dxdt\right)^{1/2}.
\]
\end{lemma} 
\noindent
Our next lemma confirms that we can approximate in $L^2(Q_{7/4})$ the solution $u$ of \eqref{u-Q5} by the solution $v$ of \eqref{v-Q4} if the coefficients and the data are sufficiently close to each others.
\begin{lemma} \label{L2-aprox} Let $M_0, \Lambda >0$  and $s >1$, be fixed. Then, for every $\epsilon >0$, there exists $\delta >0$ depending on only $\epsilon, \Lambda, n, M_0, s$ such that the following statement holds true: For every $a, b, F$ such that if \eqref{ellipticity} holds,  $\norm{b}_{L^\infty(\Gamma_2, \mathcal{V}^{1,2}(B_2))} \leq M_0$, , and 
\begin{equation} \label{delta-smallness-ass}
\begin{split}
 & \left\{  \fint_{Q_{7/4}} |a -\bar{a}_{B_{7/4}}(t)|^2 dxdt \right\}^{1/2}
 + \left\{ \fint_{Q_{2}} |F|^2 dxdt  \right\}^{1/2} \\
 & \quad + \left\{\fint_{Q_{2}} |b|^{s} dxdt \right\}^{1/s} \left\{ \fint_{Q_{2}}|\hat{u}|^{s'} dx dt \right\}^{1/s'} \leq \delta 
\end{split} 
\end{equation}
with $\hat{u} = u - \bar{u}_{Q_{2}}$, then  every weak solution $u$ of \eqref{u-Q5} with
\[
\fint_{Q_{2}} |\nabla u|^2 dxdt \leq 1,
\]
the weak solution $v$ of \eqref{v-Q4} satisfies
\[
\fint_{Q_{7/4}} |u - v|^2 dxdt  \leq \epsilon, \quad \text{and} \quad \fint_{Q_{7/4}}|\nabla v|^2dxdt \leq C(\Lambda, M_0, n).
\]
\end{lemma}
\begin{proof} Note that once the existence is proved, it follows from  Lemma \ref{w-energy} and the assumption \eqref{delta-smallness-ass} that
\[
\int_{Q_{7/4}} |\nabla w|^2 dx dt \leq  C[M_0 +1].
\]
From this, and using \eqref{delta-smallness-ass}, we infer that
\[
\fint_{Q_{7/4}} |\nabla v|^2 dx dt \leq \fint_{Q_{7/4}} |\nabla w|^2 dx dt + \fint_{Q_{7/4}} |\nabla u|^2 dxdt \leq C(\Lambda, M_0,n).
\]
 Therefore, we only need to prove the existence of $\delta$. We use the contradiction argument as this method works well for nonlinear equations, and non-smooth domains. Assume that there exist $M_0, \Lambda >0, s, s', \lambda$, and $\epsilon_0 >0$ be as in the assumption such that for every $k \in \mathbb{N}$, there are $F_k, a_k, b_k, $ such that 
\begin{equation} \label{a_k} 
\begin{split} 
& \left\{  \fint_{Q_{7/4}} |a_k -\bar{a}_{k, B_{7/4}}(t)|^2 dxdt \right\}^{1/2}
 + \left\{ \fint_{Q_{2}} |F_k|^2 dxdt  \right\}^{1/2} \\
 &  \quad + \left\{\fint_{Q_{2}} |b_k|^{s} dxdt \right\}^{1/s} \left\{ \fint_{Q_{2}}|\hat{u}_k|^{s'} dx dt \right\}^{1/s'} \leq \frac{1}{k},  \quad \text{for} \quad \hat{u}_k = u_k - \bar{u}_{k, Q_{2}}
\end{split} 
\end{equation}
and a weak solution $u_k$ of 
\begin{equation} \label{u-k.eqn}
\partial_t u_k - \text{div}[a_k \nabla u_k] -   b_k \cdot \nabla u_k = \text{div}(F_k), \quad Q_2,
\end{equation}
satisfying 
\begin{equation} \label{gradient-b-k}
 \fint_{Q_2} |\nabla u_k|^2 dxdt \leq 1,
\end{equation}
but for the weak solution $v_k$ of
\begin{equation} \label{vk-Q4}
\left\{
\begin{array}{cccl}
\partial_tv_k - \text{div}[ \bar{a}_{k, B_{7/4}}(t) \nabla v]  &= & 0, & \quad \text{in} \quad Q_{7/4},\\
v_k & = & u_k, & \quad \text{on} \quad \partial_p Q_{7/4},
\end{array} \right.
\end{equation}
we have 
\begin{equation} \label{epsilon-0}
\fint_{Q_{7/4}} |u_k - v_k|^2 dxdt \geq \epsilon_0.
\end{equation}
Since $\bar{a}_{k, B_{7/4}}(t)$ is a bounded sequence in $L^\infty(\Gamma_{7/4}, \mathbb{R}^{n\times n})$, we can also assume that there is $\bar{a}(t)$ in $L^\infty(\Gamma_{7/4}, \mathbb{R}^{n\times n}))$ such that $\bar{a}_{k, B_{7/4}} \rightharpoonup  \bar{a}$ weakly-* in $L^\infty(\Gamma_{7/4}; \mathbb{R}^{n\times n})$. This means that for each vector $\xi\in \R^n$, and for all  function $\phi\in L^1(\Gamma_{7/4})$, we have
\begin{equation} \label{a-converge}
\int_{\Gamma_{7/4}} \wei{ \bar{a} (t) \xi, \xi} \phi(t) dt =\lim_{k\to\infty} \int_{\Gamma_{7/4}} \wei{ \bar{a}_{k,B_{7/4}}(t) \xi, \xi} \phi(t) dt.
\end{equation}
Also, for each $k~\in~\mathbb{N}$, let $w_k = u_k - v_k$, we see that $w_k$ is a weak solution of 
\begin{equation} \label{w-k-eqn}
\left\{
\begin{array}{cccl}
\partial_t w_k - \text{div} [\bar{a}_{k, B_{7/4}} \nabla w_k + (a_k -\bar{a}_{k, B_{7/4}}) \nabla u_k]  - b_k\cdot  \nabla u_k  & = & \text{div}[F_k], & \quad Q_{7/4}, \\
w_k  & = & 0, & \quad \partial_p Q_{7/4}.
\end{array}
\right.
\end{equation}
From \eqref{a_k}, and \eqref{gradient-b-k}, we can apply Lemma \ref{w-energy} to yield
\begin{equation} \label{w-k-bounded}
\sup_{\Gamma_{7/4}} |w_k|^2 dx + \int_{Q_{7/4}} |\nabla w_k|^2 dx dt \leq C, \quad \forall k \in \mathbb{N}.
\end{equation}
This estimate, together with \eqref{a_k}, \eqref{gradient-b-k}, and the PDE in \eqref{w-k-eqn}, we conclude that $\{w_k\}_k$ is a bounded sequence in $\mathcal{E}(Q_{7/4})$, where
\[\mathcal{E}(Q_{7/4}) = \{g \in L^2(\Gamma_{7/4}, H^1(B_{7/4})): g_t \in L^2(\Gamma_{7/4}, H^{-1}(B_{7/4}), g =0 \ \text{on} \ \partial_p Q_{7/4} \}.
\]
Therefore, by the compact embedding $\mathcal{E}(Q_{7/4}) \hookrightarrow C(\overline{\Gamma}_{7/4}, L^2(B_{7/4}))$, and by passing through a subsequence, we can assume that there is $w \in \mathcal{E}(Q_{7/4})$ such that
\begin{equation} \label{w-k-converge}
 \left \{ \  
\begin{array}{lll}
&w_k \to w  \mbox{ strongly in } L^2(Q_{7/4}),\quad \nabla w_k \rightharpoonup \nabla w\text{ weakly in } L^2(Q_{7/4}),\\
&\partial_t w_k \rightharpoonup \partial_t  w \text{ weakly-* in } L^2(\Gamma_{7/4}; H^{-1}(B_{7/4})), \quad \text{and} \quad w_k \rightarrow w \ \text{a.e. in} \ Q_{7/4}.
\end{array} \right . 
\end{equation}
From \eqref{epsilon-0} and \eqref{w-k-converge}, it follows that
\begin{equation} \label{w-epsilon-0}
\fint_{Q_{7/4}} w^2 dxdt \geq \epsilon_0.
\end{equation}
Moreover, due to the boundary condition $w_k = 0$ on $\partial_p Q_{7/4}$, and \eqref{w-k-converge}, we also conclude that, in the trace sense,
\begin{equation} \label{w-boundary}
w =0, \quad \partial_p Q_{7/4}.
\end{equation}
We claim that $w$ is a weak solution of 
\begin{equation} \label{w-unbounded-lambda}
\left \{
\begin{array}{cccl}
w_t -\text{div}[\bar{a}(t) \nabla w]  &= &0, & \quad  Q_{7/4}, \\
 w & = &0, & \quad\partial_p Q_{7/4} 
\end{array} \right.
\end{equation}
From this, and by the uniqueness of the weak solution of this equation, we infer that $w =0$ and this contradicts to \eqref{w-epsilon-0}. Thus, it remains to prove that 
$w$ is a weak solution of \eqref{w-unbounded-lambda}.  To prove this, we pass the limit as $k\rightarrow \infty$ of \eqref{w-k-eqn}. By \eqref{w-boundary}, we only need to find the limits as $k\rightarrow \infty$ for each term in the weak form of the equation \eqref{w-k-eqn}.  Let us fix a test function $\varphi \in C^\infty(\overline{Q}_{7/4})$ with $\varphi =0$ on $\partial_pQ_{7/4}$. Then,  it is easy to see from \eqref{a_k}, and \eqref{gradient-b-k} that 
\[
\lim_{k\rightarrow \infty} \int_{Q_{7/4}} F_k \cdot \nabla \varphi  dx dt =0, \quad \lim_{k\rightarrow \infty} \int_{Q_{7/4}} \wei{(a_k -\bar{a}_{k, B_{7/4}}(t)) \nabla u_k, \nabla \varphi} dxdt =0.
\]
Further more, from \eqref{w-k-converge}, we also find that
\[
\lim_{k\rightarrow \infty} \int_{\Gamma_{7/4}}\ \wei{\partial_t w_k,\varphi}_{H^{-1}(B_{7/4}), H^1_0(B_{7/4})} dt =\int_{\Gamma_{7/4}}\ \wei{\partial_t w,\varphi}_{H^{-1}(B_{7/4}), H^1(B_{7/4})} dt.
\]
For the term involving $b_k$, since $\text{div} (b_k) =0$,  we can use the integration by parts in $x$  to write
\[
 \int_{Q_{7/4}} [b_k \cdot \nabla u_k]  \varphi  dxdt = - \int_{Q_{7/4}}  \Big[b\cdot \nabla \varphi \Big] \hat{u}_k dxdt, \quad \hat{u}_k = u_k - \bar{u}_{k, Q_{2}}.
\]
Then, by H\"{o}lder's inequality and \eqref{gradient-b-k},  see that
\[
\begin{split}
\left| \int_{Q_{7/4}}[ b_k \cdot \nabla u_k] \varphi  dxdt  \right| & \leq  \norm{ \nabla \varphi}_{L^\infty(Q_{7/4})}\left\{ \int_{Q_{2}} |b_k|^{s} dxdt  \right\}^{1/s} \left\{ \int_{Q_{2}} |\hat{u}_k|^{s'} dxdt \right\}^{1/s'} \\
& \leq \frac{|Q_{7/4}|}{k^{1/2}}  \norm{ \nabla \varphi}_{L^\infty(Q_{7/4})} \rightarrow 0, \quad \text{as} \quad k \rightarrow \infty.
\end{split}
\]
Finally, since $\bar{a}_{k, B_{7/4}}$ and $\bar{a}$ are independent on $x$, by integrating by parts in $x$, we have
\[
\begin{split}
& \int_{Q_{7/4}} \Big[\wei{\bar{a}_{k,B_{7/4}}(t) \nabla w_k, \nabla \varphi}  - \wei{\bar{a} (t) \nabla w, \nabla\varphi} \Big] dxdt  \\
& =- \sum_{i,j=1}^n \int_{Q_{7/4}} \Big[ w_k\bar{a}_{k,B_{7/4}}^{ij} (t)\partial_{x_ix_j} \varphi  -  w \bar{a}^{ij} (t)\partial_{x_ix_j} \varphi \Big] dxdt \\
& = -\sum_{i,j=1}^n  \int_{Q_{7/4}} \left\{ \bar{a}_{k, B_{7/4}}^{ij}(t) \partial_{x_ix_j} \varphi \Big[ w_k -w\Big] + w\partial_{x_i x_j} \varphi \Big[\bar{a}_{k, B_{7/4}}^{ij}(t) - \bar{a}^{ij}(t) \Big] \right\} dxdt
\end{split}
\]
Hence, it follows from \eqref{a-converge} and \eqref{w-k-converge} that
\[
\lim_{k\rightarrow \infty} \int_{Q_{7/4}} \Big[\wei{\bar{a}_{k,B_{7/4}}(t) \nabla w_k, \varphi}  - \wei{\bar{a} (t) \nabla w, \nabla\varphi} \Big] dxdt =0.
\]
Collecting the efforts, we obtain
\[
\int_{\Gamma_{7/4}}\wei{w_t, \varphi}_{H^{-1}(B_{7/4}), H^1_0(B_{7/4})} dt + \int_{Q_{7/4}} \wei{\bar{a}(t)\nabla w, \nabla \varphi} dxdt =0, \quad \forall \ \varphi \in C^\infty(\overline{Q}_{7/4}): \ \varphi =0 \ \text{on} \ \partial_p Q_{7/4}.
\]
Thus,  $w$ is a weak solution of \eqref{w-unbounded-lambda}. The proof is then complete.
\end{proof}
\begin{lemma} \label{gradient-approx} Let $M_0, s >0$, and $\Lambda >0$ be fixed. Then, for every $\epsilon >0$, there exists $\delta >0$ depending on $n, M_0, \Lambda, s$, and $\epsilon$ such that the following statement holds true: For every $a, b, F$ such that if \eqref{ellipticity} holds,  $\norm{b}_{L^\infty(\Gamma_2, \mathcal{V}^{1,2}(B_2))} \leq M_0$, and
\[
\begin{split} 
 \left\{  \fint_{Q_{7/4}} |a -\bar{a}_{B_{7/4}}(t)|^2 dxdt \right\}^{1/2}
 + \left\{ \fint_{Q_{2}} |F|^2 dxdt  \right\}^{1/2}+ \left\{\fint_{Q_{2}} |b|^{s} dxdt \right\}^{1/s} \left\{ \fint_{Q_{2}}|\hat{u}|^{s'} dx dt \right\}^{1/s'} \leq \delta 
 \end{split}  
\]
then, for every weak solution $u$ of \eqref{u-Q5} with
\[
\fint_{Q_{2}} |\nabla u|^2 dxdt \leq 1, 
\]
the weak solution $v$ of \eqref{v-Q4} satisfies
\[
\fint_{Q_{3/2}} |\nabla u -\nabla v |^2 dxdt\leq \epsilon.
\]
Moreover, there is $C = C(\Lambda,M_0, n)$ such that
\begin{equation} \label{gradient-v-est}
\norm{\nabla v}_{L^\infty(Q_{\frac{3}{2}})} \leq C(n, \Lambda, M_0).
\end{equation}
\end{lemma}
\begin{proof}  Let $\mu >0$ to be determined. By Lemma \ref{L2-aprox}, there exists $\delta_1 >0$ such that if $\norm{b}_{L^\infty(\Gamma_2, \mathcal{V}^{1,2}(B_2))} \leq M_0$, and 
\[
\begin{split} 
 \left\{  \fint_{Q_{7/4}} |a -\bar{a}_{B_{7/4}}(t)|^2 dxdt \right\}^{1/2}
 + \left\{ \fint_{Q_{2}} |F|^2 dxdt  \right\}^{1/2}+ \left\{\fint_{Q_{2}} |b|^{s} dxdt \right\}^{1/s} \left\{ \fint_{Q_{2}}|\hat{u}|^{s'} dx dt \right\}^{1/s'} \leq \delta_1,
 \end{split} 
\]
then
\begin{equation} \label{u-v-aproximat}
\fint_{Q_{7/4}}|u - v|^2 dx dt \leq \mu, \quad \fint_{Q_{7/4}} |\nabla v|^2 dxdt \leq C(\Lambda, M_0).
\end{equation}
where $u$ is a weak solution of \eqref{u-Q5}, and $v$ is a weak solution of \eqref{v-Q4} and
\[
\fint_{Q_2} |\nabla u|^2 dxdt \leq 1.
\]
From \eqref{u-v-aproximat} and Lemma \ref{reg-v}, we con conclude that
\[
\norm{\nabla v}_{L^\infty(Q_{\frac{3}{2}})} \leq C(n, \Lambda, M_0).
\]
Note that without loss of generality, we can assume that $\delta_1 \leq \mu$. Thereofore, applying Lemma \ref{w-local-energy}, we obtain
\[
\fint_{Q_{3/2}} |\nabla u - \nabla v|^2 dxdt \leq C(\Lambda, M_0, n) \mu^{1/2}.
\]
Therefore, if we choose $\mu$ such that $\mu^{1/2} = \epsilon/ C(\Lambda, M_0, n)$, the lemma follows.
\end{proof}
We in fact need a localized version of Lemma \ref{gradient-approx}. For each $r>0$ and $z_0 = (x_0, t_0) \in Q_1$, we approximate a weak solution of the equation
\begin{equation} \label{u-Q-5r}
u_t - \text{div}[a \nabla u] - b\cdot \nabla u = \text{div}(F), \quad \text{in} \quad  Q_{2r}(z_0),
\end{equation}
by the weak solution of
\begin{equation} \label{v-Q-4r}
\left\{
\begin{array}{cccl}
v_t - \text{div}[\bar{a}_{B_{7r/4}(x_0)}(t) \nabla v] & = &0, & \quad \text{in} \quad Q_{7r/4}(z_0), \\
v & = & u, & \quad \text{on}\quad \partial_p Q_{7r/4}(z_0). 
\end{array} \right.
\end{equation}
We then have the following lemma, which is the main result of the section.
\begin{lemma}\label{localized-compare-solution} Let $\Lambda >0, s >1$, and $M_0 >0$ be fixed. Then, for any $\e>0$, there exists $\delta>0$ depending only on $\e$, $\Lambda, M_0, s$ and  $n$  such that the following statement holds true: For every $z_0 \in Q_1$, $0<r\leq 1$, and for every $a, b, F$ such that  \eqref{ellipticity} holds for $a$, and  
\begin{equation} \label{smallness-condition-c} 
\begin{split}
& \left\{ \fint_{Q_{7r/4}(z_0)} |a -\bar{a}_{B_{7r/4}(x_0)}|^2  dxdt \right\}^{1/2}
 + \left\{ \fint_{Q_{2r}(z_0)} |F|^2 dxdt \right\}^{1/2} \\
 & \quad  + \left\{ \fint_{Q_{2r}(z_0)} |u - \bar{u}_{Q_{2r}(z_0)}|^{s'} \right\}^{1/s'}\left\{ \fint_{Q_{2r}(z_0)} |b|^{s} dxdt \right\}^{1/s}  \leq \delta
\end{split}
 \end{equation}
 then for every weak solution $u$ of \eqref{u-Q5} with
\begin{equation} \label{3.6-1}
 \fint_{Q_{2r}(z_0)}{|\nabla u|^2 \, dxdt}\leq 1,
\end{equation}
 the  weak solution  $v$ of \eqref{v-Q-4r} satisfies   
\beq \label{compare-graduv}
\fint_{Q_{3r/2}(z_0)}{|\nabla u - \nabla v|^2\, dx dt}\leq \e,
\quad \text{and} \quad \norm{\nabla v}_{L^\infty({Q_{3r/2}(z_0)})} \leq C(\Lambda, M_0, n).
\end{equation}
\end{lemma}
\begin{proof} Given any $\e>0$, let $\delta=\delta(\e,\Lambda, M_0, s, n)>0$ be defined as in  Lemma~\ref{gradient-approx}. We now show that Lemma \ref{localized-compare-solution} holds with this $\delta$. Let $u, v$ satisfy the conditions in of Lemma \ref{localized-compare-solution}. Without loss of generality, we can assume that $z_0 = (0,0)$. Let us define
\[
u'(x,t) =\frac{u(r x, r^2 t)}{r}, \quad v'(x,t)= \frac{v(r x, r^2 t)}{r}, \quad a'(x, t) = a(r x, r^2 t). \]
Also, let us denote
\[
F'(x,t) = F(rx, r^2t),  \quad b'(x,t) = r b(rx, r^2t).
\]
Then $u'$ is a weak solution of
\begin{equation*}
u'_t  -  \text{div}[a' \nabla u']  - b'\cdot \nabla u'= \text{div}(F')  \quad\mbox{in}\quad Q_2
\end{equation*}
and $v'$ is  a weak solution of  
\begin{equation*}
\left \{
\begin{array}{cccl}
v'_t  &=&  \text{div} [\bar{a'}_{B_{7/4}}(t)\nabla v']  \quad &\text{in}\quad Q_{7/4}, \\
v' & =& u'\quad &\text{on}\quad \partial_p Q_{7/4}.
\end{array}\right.
\end{equation*}
We now check that the conditions in Lemma \ref{gradient-approx} hold with $a', b', u', F'$ and $v'$. A simple calculation shows 
\[
\begin{split}
&   \norm{b'}_{L^\infty(\Gamma_2, \mathcal{V}^{1,2}(B_2))} = \norm{b}_{L^\infty(\Gamma_{2r}, \mathcal{V}^{1,2}(B_{2r}))}  \leq M_0, \\
& \fint_{Q_{7/4}} \Big[|a' -\overline{a'}_{B_{3}}|^2 dxdt = \fint_{Q_{7r/4}} |a' -\bar{a}_{B_{7r/4}}|^2 dxdt, \quad \fint_{Q_{2}} |F'|^2 dx dt  =  \fint_{Q_{2r}} |F|^2 dxdt. 
\end{split}
\]
Also,  
\[
\begin{split}
& \fint_{Q_2} |\nabla u'|^2 dxdt= \fint_{Q_{2r}} |\nabla u|^2 dxdt \leq 1, \\
& \left\{ \fint_{Q_{2}} |u' - \bar{u}'_{Q_{2}}|^{s'} \right\}^{1/s'}\left\{ \fint_{Q_{2}} |b'|^{s} dxdt \right\}^{1/s}  =  
\left\{ \fint_{Q_{2r}} |u - \bar{u}_{Q_{2r}}|^{s'} \right\}^{1/s'}\left\{ \fint_{Q_{2r}} |b|^{s} dxdt \right\}^{1/s}.
\end{split}
\]
Therefore, if  \eqref{smallness-condition-c} and \eqref{3.6-1} hold, then all conditions in Lemma \ref{gradient-approx} are met. Hence, we have
\begin{equation*}
\fint_{Q_{3/2}}{|\nabla u'(x,t) - \nabla v'(x,t)|^2\, dx dt}\leq \e, \quad \text{and} \quad 
\norm{\nabla v'}_{L^\infty(Q_{3/2})} \leq C(\Lambda,n).
\end{equation*}
By a simple integration by substitution, we obtain 
\begin{equation*}
\fint_{Q_{3r/2}}{|\nabla u(x,t) - \nabla v(x,t)|^2\, dx dt}\leq \e, \quad \text{and} \quad 
\norm{\nabla v}_{L^\infty(Q_{3r/2} )} \leq C(\Lambda,n).
\end{equation*}
The proof is then complete.
\end{proof}
\section{Weighted density  estimates and weighted $W^{1,p}$-regularity estimates}\label{interior-density-gradient} \label{density-est-section}
\subsection{Weighted density estimates} We will derive the estimate of $\norm{\nabla u}_{L^p(Q_1, \omega)}$ for solution $u$ of \eqref{u-Q5} by estimating the distribution functions of the maximal function of $|\nabla u|^2$. Our first lemma gives a density estimate for the distribution  of $\M_{Q_2}(|\nabla u|^2)$, where the maximal operator $\M_{Q_2}$ is defined in Definition \ref{Maximal-Operator}. From now let us fixe $s, s' \in (1, \infty), \alpha >0$ and $\lambda$ satisfying \eqref{s-s'-alpha}.
If $u$ is a weak solution of \eqref{u-Q5}, we define
\begin{equation} \label{B.def}
B(x,t) = \big( [u]_{s', \lambda, Q_1} |b(x,t)| \big)^s,
\end{equation}
where $[u]_{s', \lambda, Q_1}$ is the parabolic semi-Campanato's norm of $u$ on $Q_1$,
\[
[u]_{s', \lambda, Q_1} = \sup_{0 <\rho< 1, z \in Q_1} \left[\rho^{-\lambda} \fint_{Q_\rho(z)} |u(x,t) - \bar{u}_{Q_\rho(z)}|^{s'} dxdt\right]^{1/s'}.
\]
\begin{lemma}\label{initial-density-est} Let $\Lambda >0, M_0 >0$ be fixed, and $\omega \in A_q$ for some $1<q <\infty$. Let $s, \lambda, \alpha$ be as in \eqref{s-s'-alpha}. Then, there exists a constant $N>1$ depending only on $\Lambda, M_0, s, \lambda$ and $n$ such that  for  every $\e>0$, we can find  $\delta~=~\delta(\e,M_0,\Lambda, s, \lambda, [\omega]_{A_q}, n)>0$  such that the following statement holds true:  Suppose that \eqref{ellipticity} holds for the matrix $a$, $\textup{div}(b) =0$, and $|F|, |b| ~\in~L^2(Q_2)$.  
\begin{equation}\label{interior-SMO}
\sup_{0<\rho\leq 1}\sup_{(y,s)\in Q_1} \fint_{Q_\rho(y, s)}{|a(x, t) - \bar{a}_{B_\rho(y)}(t)|^2\, dx dt}\leq \delta,  \quad \norm{ b}_{L^\infty(\Gamma_2, \mathcal{V}^{1,2}(B_2))}\leq M_0,
\end{equation}
and for a  weak solution $u$ of \eqref{u-Q5} 
 and for every $z= (y,\tau)\in Q_1$, and $0~<~r~\leq~1/2$ if the set
\begin{equation*}
\begin{split}
  Q_r(z) \cap Q_1\cap \big\{Q_2:\, \M_{Q_2}(|\nabla u|^2)\leq 1 \big\} \cap
  \{ Q_2: \M_{Q_2}(|F|^2)  +  \M_{\alpha, Q_2}(B)^{2/s} \leq \delta\},
 \end{split}
\end{equation*}
is non-empty, then 
\begin{equation*}
 \omega \big( \{ Q_1:\, \M_{Q_2}(|\nabla u|^2)> N \}\cap Q_r(z)\big)
\leq \e  \omega\big( |Q_r(z)\big).
\end{equation*} 
\end{lemma}
\begin{proof} Let $\eta >0$ depending only on $\epsilon, \Lambda, M_0, s, [\omega]_{A_q}$ and $\lambda$ to be determined. Then, let $\delta = \delta(\eta, \Lambda, M_0, s, n)$ be the number defined in Lemma \ref{localized-compare-solution}. We prove our lemma with this choice of $\delta$. By the condition on the non-empty intersection, there exists a point $z_0 = (x_0,t_0)\in Q_r(z)\cap  Q_1$ such that
\begin{equation}\label{maximal-fns-control}
\begin{split}
& \M_{Q_2}(|\nabla u|^2)(z_0)\leq 1,\quad \mbox{and}\quad  \M_{Q_2}(|F|^2)(z_0) +\M_{\alpha, Q_2}(B)(z_0)^{2/s} \leq
\delta.
\end{split}
\end{equation}
Notice that with $r \in (0,1/2)$, $Q_{2 r}(z) \subset Q_2$. Since $Q_{2 r}(z) \subset Q_{3 r}(z_0) \cap Q_2$, it follows from \eqref{maximal-fns-control} that
\begin{equation*}
 \fint_{Q_{2 r}(z)}|\nabla u|^2 \, dx dt \leq \frac{|Q_{3r}(z_0)|}{|Q_{2r}(y,s)|} \frac{1}{|Q_{3r}(z_0)|}\int_{Q_{3 r}(z_0)\cap Q_2}|\nabla u|^2 \, dx dt\leq \Big(\frac{3}{2}\Big)^{n+2}.
\end{equation*}
Similarly,
\begin{equation*}
\begin{split}
& \left\{\fint_{Q_{2r}(z)} |u - \bar{u}_{Q_{2r}(z)}|^{s'} dxdt \right\}^{2/s'} \left \{ \fint_{Q_{2r}(z)} |b|^s dxdt \right\}^{2/s} \\
& = \left\{(2r)^{-\lambda} \fint_{Q_{2r}(z)} |u - \bar{u}_{Q_{2r}(z)}|^{s'} dxdt \right\}^{2/s'} \left \{ (2r)^{\alpha}\fint_{Q_{2r}(z)} |b|^s dxdt \right\}^{2/s} \\
& \leq [u]_{s', \lambda, Q_1}^2 \left \{ (2r)^\alpha \frac{|Q_{3r}(z_0)|}{|Q_{2r}(z)|} \fint_{Q_{3r}(z_0) \cap Q_2} |b|^s dxdt \right\}^{2/s} \\\
& \leq \big(3/2 \big)^{2(n+2-\alpha)/s} \left \{ (3r)^\alpha  \fint_{Q_{3r}(z_0) \cap Q_2} |B|^s dxdt \right\}^{2/s} \\
& \leq \big(3/2 \big)^{2(n+2-\alpha)/s}\M_{\alpha, Q_2} (B)^{2/s} \leq \big(3/2 \big)^{2(n+2-\alpha)/s} \delta,
\end{split}
\end{equation*}
where we have used \eqref{s-s'-alpha} in our second step in the above estimate. Moreover, we also have
and
\begin{align*}
& \fint_{Q_{3 r}(z)}|F|^2  dxdt \leq  \frac{|Q_{3 r}(z_0)|}{|Q_{2 r}(z)|} \frac{1}{|Q_{3 r}(z_0)|} 
 \int_{Q_{3r}(z_0)\cap Q_2} |F|^2 \, dx dt  \leq \Big(\frac{3}{3}\Big)^{n+2} \delta.
 \end{align*}
Also from the assumption \eqref{interior-SMO}, and since $Q_{7r/4}(z) \subset Q_2$, we also have
\begin{align*}
 & \fint_{Q_{7r/4}(z)} |a(x,t)-\bar{a}_{B_{7r/4}(y)}(t) |^2 \, dx dt  \leq  \delta, \\
 & \norm{b}_{L^\infty(\Gamma_{2r}(\tau), \mathcal{V}^{1,2}(B_{2r}(y)))} \leq  \norm{b}_{L^\infty(\Gamma_2, \mathcal{V}^{1,2}(B_2))} \leq M_0.
\end{align*}
These estimates together  allow us to use Lemma~\ref{localized-compare-solution} with a suitable scaling to obtain
\begin{equation}\label{gradient-comparison}
\fint_{Q_{3r/2}(z)}{|\nabla u- \nabla v|^2\, dx dt}\leq \kappa \eta, \quad \norm{\nabla v}_{L^\infty(Q_{3r/2}(z))}  \leq C_0 := \kappa C(\Lambda, M_0, n).
\end{equation}
where 
\[
\kappa = \max\Big\{(3/2)^{n+2},  \big(3/2 \big)^{2(n+2-\alpha)/s}\Big\},
\]
and $v$ is the  unique weak solution of 
\begin{equation*}
\left \{
\begin{array}{cccl}
v_t  &=&  \nabla\cdot[\bar{a}_{B_{7r/4}(y)}(t)\nabla v]  \quad &\text{in}\quad Q_{7r/4}(z), \\
v & =& u\quad &\text{on}\quad \partial_p Q_{7r/4}(z)
\end{array}\right.
\end{equation*}
We claim that \eqref{maximal-fns-control}, and \eqref{gradient-comparison} 
yield
\begin{equation}\label{set-relation-claim}
 \big\{ Q_r(z):  \M_{Q_{3r/2}(z)}(|\nabla u - \nabla v|^2) \leq C_0  \big\}\subset \big\{ Q_r(z):\, \M_{Q_2}(|\nabla u |^2) \leq N\big\}
\end{equation}
with $N = \max{\{4 C_0, 5^{n+2}\}}$. Indeed, let $(x,t)$ be a point in the set on the left hand side of \eqref{set-relation-claim}, and consider 
$Q_\rho(x,t)$. If $\rho\leq r/2$, then $Q_\rho(x,t) \subset Q_{\frac{3r}{2}}(z)\subset Q_{2}$
and hence
\begin{align*}
 &\frac{1}{|Q_\rho(x,t)|}\int_{Q_\rho(x,t) \cap Q_2} |\nabla u|^2 \, dx dt\\
 &\leq 
 \frac{2}{|Q_\rho(x,t)|}\left[\int_{Q_\rho(x,t) \cap Q_2} |\nabla u-\nabla v|^2 \, dx dt
 +\int_{Q_\rho(x,t) \cap Q_2} |\nabla v|^2 \, dx dt \right]\\
 &\leq 2 \M_{Q_{3r/2}(z)}(|\nabla u - \nabla v|^2)(x,t) +2 \|\nabla v \|_{L^\infty(Q_{\frac{3r}{2}}(z))}^2
 \leq 4C_0(\Lambda, M_0, n).
\end{align*}
On the other hand if $\rho>r/2$, then  $Q_\rho(x,t)\subset Q_{5\rho}(z_0)$. This and  the first inequality in \eqref{maximal-fns-control} imply that
\begin{align*}
 \frac{1}{|Q_\rho(x,t)|}\int_{Q_\rho(x,t) \cap Q_2} |\nabla u|^2 \, dx dt
 \leq \frac{5^{n+2}}{|Q_{5\rho}(x_0, t_0)|} \int_{Q_{5\rho}(z_0) \cap Q_2} |\nabla u|^2 \, dx dt\leq 
 5^{n+2}.
\end{align*}
Therefore, $\M_{Q_2}(|\nabla u |^2)(x,t) \leq N$ and the claim \eqref{set-relation-claim} is proved.
Note that   \eqref{set-relation-claim} is equivalent to
 \begin{align*}
 E: = \big\{Q_r(z):\, \M_{Q_2}(|\nabla u |^2) > N\big\} \subset
 \big\{ Q_r(z):\, \M_{Q_{3r/2}(z)}(|\nabla u - \nabla v|^2) >C_0 \big\}.
 \end{align*}
It follows from this,  the weak type $1-1$ estimate of the Hardy-Littlewood maximal function, and \eqref{gradient-comparison}   that
\begin{align*}
 & \big |E \big |\leq  \big |
 \big\{ Q_r(z):\, \M_{Q_{3r/2}(z)}(|\nabla u - \nabla v|^2) > C_0\big\}\big| \\
 &\leq \frac{C(n) |Q_{3r/2}(z)|}{C_0}  \fint_{Q_{3r/2}(z)}{|\nabla u - \nabla v|^2 \, dx dt}\leq C' \eta \, |Q_r(z)|,
 \end{align*}
 where $C'>0$ depends only on $\Lambda, M_0, s, \alpha, $ and $n$. Then, from Lemma \ref{doubling}, there is 
 $\beta = \beta([\omega]_{A_q}, n) >0$ such that
 \[
 \omega(E) \leq C([\omega]_{A_q}, n) \Big( \frac{|E|}{|Q_r(z)|} \Big)^{\beta} \omega(Q_r(z)) \leq C_*  \eta^\beta \omega(Q_r(z)),
 \]
where $C_*>0$ is a constant depending only on $\Lambda, M_0, s, \alpha, [\omega]_{A_q} $ and $n$.
By choosing $\eta = \Big(\frac{\e}{C_*}\Big)^{1/\beta}$, we obtain the desired result.
\end{proof}
\begin{lemma}\label{second-density-est} Let $\Lambda >0, M_0 >0$ be fixed and $\omega \in A_q$ with some $1 < q < \infty$. L et $s, \lambda, \alpha$ be as in \eqref{s-s'-alpha}. There exists a constant $N>1$ depending only on $\Lambda, M_0, s, \lambda$ and  $n$ such that   
for  any $\e>0$, we can find  $\delta=\delta(\e, \Lambda, M_0, s, \lambda, [\omega]_{A_q}, n)>0$  such that if  \eqref{ellipticity} holds for the matrix $a$, $\textup{div}(b) =0$, and $|F|, |b| ~\in~L^2(Q_2)$.
\begin{equation*}
\sup_{0<\rho\leq 1}\sup_{(y,s)\in Q_1} \fint_{Q_\rho(y,s)}{|a(x, t) - \bar{a}_{B_\rho(y)}(t)|^2\, dx dt}\leq \delta, \quad \norm{b}_{L^\infty(\Gamma_2, \mathcal{V}^{1,2}(B_2))} \leq M_0, 
\end{equation*}
 and if a weak solution $u$ of \eqref{u-Q5} satisfying
\begin{equation*}
\begin{split} 
 \omega \big ( \{Q_1: 
\M_{Q_2}(|\nabla u|^2)> N \}\big| \leq \e \omega(Q_1(z) \big), \quad \forall \ z \in Q_1.
\end{split}
\end{equation*}
Then it holds that 
\begin{align*}
\omega\big(\{Q_1: \M_{Q_2}(|\nabla u|^2)> N\}\big)
& \leq \e_1  \, \Big\{
\omega\big(\{Q_1: \M_{Q_2}(|\nabla u|^2)> 1\}\big) \\
& \quad 
+ \omega\big(\{ Q_1:\, \M_{Q_2}(|F|^2)  + \M_{\alpha, Q_2}(B)^{2/s} > \delta \}\big)\Big\}, \nonumber
\end{align*}
where $\e_1$ is defined in Lemma \ref{Vitali}.
\end{lemma}
\begin{proof} In view of Lemma~\ref{initial-density-est}, we can apply Lemma \ref{Vitali},  for 
 \[ E=\{ Q_1:\, \M_{Q_2}(|\nabla u|^2)> N \} \]
 and 
\[ K=\{ Q_1:\, \M_{Q_2}(|\nabla u|^2)> 1 \}\cup \{ Q_1:\, \M_{Q_2}(|F|^2) + \M_{\alpha, Q_2}(B)^{2/s} > \delta \}
\]
 to obtain the desired estimate.
\end{proof}
\begin{lemma} \label{iteration-lemma}
 Let $\Lambda >0, M_0 >0$ be fixed and $\omega \in A_q$ with some $1 < q < \infty$. L et $s, \lambda, \alpha$ be as in \eqref{s-s'-alpha}. There exists a constant $N>1$ depending only on $\Lambda, M_0, s, \lambda$ and  $n$ such that   
for  any $\e>0$, we can find  $\delta=\delta(\e, \Lambda, M_0, s, \lambda, [\omega]_{A_q}, n)>0$  such that if  \eqref{ellipticity} holds for the matrix $a$, $\textup{div}(b) =0$, and $|F|, |b| ~\in~L^2(Q_2)$.
\begin{equation*}
\sup_{0<\rho\leq 1}\sup_{(y,s)\in Q_1} \fint_{Q_\rho(y,s)}{|a(x, t) - \bar{a}_{B_\rho(y)}(t)|^2\, dx dt}\leq \delta, \quad \norm{b}_{L^\infty(\Gamma_2, \mathcal{V}^{1,2}(B_2))} \leq M_0, 
\end{equation*}
 and if a weak solution $u$ of \eqref{u-Q5} satisfying
\begin{equation*}
\begin{split} 
 \omega \big ( \{Q_1: 
\M_{Q_2}(|\nabla u|^2)> N \}\big| \leq \e \omega(Q_1(z) \big), \quad \forall \ z \in Q_1.
\end{split}
\end{equation*}
Then for every $k = 1,2,\cdots$, 
\begin{equation} \label{interation-est}
\begin{split}
\omega \big (\{Q_1: \M_{Q_2}(|\nabla u|^2)> N^k\}\big)
& \leq \e_1^k 
\omega \big (\{Q_1: \M_{Q_2}(|\nabla u|^2)> 1 \}\big) \\
& \qquad + \sum_{i=1}^k\e_1^i \omega \big(\{ Q_1:\, \M_{Q_2}(|F|^2) +   \M_{\alpha, Q_2}(B)^{2/s} > \delta N^{k-i} \} \big), 
\end{split}
\end{equation}
where $\e_1$ is defined in Lemma \ref{Vitali}.
\end{lemma}
\begin{proof} Let $N, \delta$ be defined as in Lemma \ref{second-density-est} and we prove \eqref{interation-est} holds with these $N, \delta$ by using induction on $k$. If $k=1$, then \eqref{interation-est} holds due to Lemma \ref{second-density-est}. Assume now that \eqref{interation-est} holds with some $m \in \N$, and we prove that it holds with for $k =m+1$. For given $u, b$ satisfying the assumptions of the lemma, we define
\[
u' = u/\sqrt{N}, \quad F' = F/\sqrt{N}, \quad B' = [u']_{s', \lambda, Q_1} |b|^{s}.
\]
Observe that $u'$ is a weak solution of
\[
\partial_t u' - \text{div}[a(x,t) \nabla u'] - b(x,t) \nabla u' = \text{div}[F], \quad \text{in} \quad Q_2.
\]
Moreover, for every $z \in Q_1$, it is simple to see that 
\[
\omega\big(Q_1:  \M_{Q_2}(|\nabla u'|^2) > N \Big) = \omega\big(Q_1:  \M_{Q_2}(|\nabla u|^2) > N^2 \big) \leq   \omega\big(Q_1:  \M_{Q_2}(|\nabla u|^2) > N \big)  \leq \e |Q_1(z)|.
\]
Therefore, we can apply the induction hypothesis for $u', F', B'$ to obtain
\[
\begin{split}
\omega \big (\{Q_1: \M_{Q_2}(|\nabla u'|^2)> N^m\}\big)
& \leq \e_1^k 
\omega \big (\{Q_1: \M_{Q_2}(|\nabla u'|^2)> 1 \}\big) \\
& \qquad + \sum_{i=1}^m\e_1^i \omega \big(\{ Q_1:\, \M_{Q_2}(|F'|^2) +   \M_{\alpha, Q_2}(B')^{2/s} > \delta N^{k-i} \} \big).
\end{split}
\]
By changing back to $u, F, B$ and using the Lemma \ref{second-density-est} again, we see that \eqref{interation-est} holds with $k = m+1$. The proof is complete.
\end{proof}
\subsection{Proof of Theorem~\ref{g-theorem}}
We are now ready to prove  Theorem~\ref{g-theorem}. The proof is quite standard once Lemma \ref{iteration-lemma} is already established, however, we give it here for completeness. Let $N = N(\Lambda, M_0, s, \lambda, n) >1$ be as in Lemma \ref{iteration-lemma}, and let $q = p/2$, and $\e_1 = (20)^{nq}[\omega]_{A_q}^2 \e$. Choose $\e$ sufficiently small depending only on $\Lambda, M_0, s, \lambda, p, n$ and $[\omega]_{A_q}$ such that
\[
N^p \e_1 < 1/2.
\] 
With this choice of $\e$, let $\delta$ be as in Lemma \ref{iteration-lemma} depending on $\Lambda, M_0, s, \lambda, p, n$ and $[\omega]_{A_q}$. We first prove Theorem \ref{g-theorem} with an additional assumption that
\begin{equation} \label{add-asump}
\omega \big(\big\{Q_1: \M_{Q_2}(|\nabla u|^2) > N \big\} \big) \leq \e \omega(Q_1(z)), \quad \forall \ z \in Q_1.
\end{equation}
Assume that \eqref{add-asump} holds, and let us denote
\[
S = \sum_{k=1}^\infty N^{kq} \omega\big(\big\{Q_1: \M_{Q_2} (|\nabla u|^2) > N  \big\}\big).
\]
By Lemma \ref{iteration-lemma}, we see that
\[
\begin{split}
S & \leq \sum_{k=1}^\infty N^{kq} \left [\sum_{i=1}^k \e_1^{i} \omega \big(\big\{ Q_1: \M_{Q_2} (|F|^2) + \M_{\alpha, Q_2} (B)^{2/s} > \delta N^{k-i} \big \}\big) \right] \\
& \quad + \sum_{k=1}^\infty N^{kq} \e_1^k  \omega\big(\big\{Q_1: \M_{Q_1} (|\nabla u|^2) \geq 1  \big\} \big).
\end{split}
\]
Then, applying the Fubini's theorem,  and Lemma \ref{measuretheory-lp}, we obtain
\[
\begin{split}
S & \leq \sum_{j=1}^\infty (N^q\e_1)^j \left[ \sum_{k=j}^\infty N^{q(k-j)} \omega\big(\big\{ Q_1: \M_{Q_2}(|F|^2) + \M_{\alpha, Q_2} (B)^{2/s} > \delta N^{k-j} \big\}\big) \right] \\
& \quad +  \sum_{k=1}^\infty N^{kq} \e_1^k  \omega\big(\big\{Q_1: \M_{Q_1} (|\nabla u|^2) \geq 1  \big\} \big) \\
& \leq C\left( \norm{\M_{Q_2}(|F|^2)}_{L^q(Q_1, \omega)}^{q} + \norm{\M_{\alpha, Q_2}(B)^{2/s}}_{L^q(Q_1, \omega)}^q +\omega(Q_1) \right) \sum_{k=1}^\infty (N^q \e_1)^k.
\end{split}
\]
Then, by our choice of $\e$, and Lemma \ref{Hardy-Max-p-p}, we obtain
\[
S \leq C \left[ \norm{F}_{L^p(Q_2,\omega)}^p +  \norm{\M_{\alpha, Q_2}(B)^{2/s}}_{L^{p/2}(Q_2, \omega)}^{p/2} + \omega(Q_1) \right].
\]
By Lemma \ref{measuretheory-lp}, it follows that
\begin{equation} \label{est-main-1}
\norm{\nabla u}_{L^p(Q_1, \omega)}^p \leq C(S + \omega(Q_1))  \leq C \left[ \norm{F}_{L^p(Q_2,\omega)}^p +  \norm{\M_{\alpha, Q_2}(B)^{2/s}}_{L^{p/2}(Q_2, \omega)}^{p/2} + \omega(Q_1) \right].
\end{equation}
Hence, we have proved \eqref{est-main-1} under the additional assumption \eqref{add-asump}. To remove \eqref{add-asump}, let us define 
\[ u' = u/\kappa, F' = F/\kappa, \quad B' = \big([u']_{s', \lambda, Q_1} |b| \big)^s,
\]
for some constant $\kappa >0$ to be determined. Observe that $u'$ is a weak solution of 
\[
\partial_t u' -\text{div}[a \nabla u'] - b\cdot \nabla u' = \text{div}(F), \quad \text{in} \quad Q_2.
\]
Let us define
\[
E:= \{Q_1: \M_{Q_2} (|\nabla u'|^2) > N  \big\} \subset Q_2.
\]
Then, it follows from Lemma \ref{doubling} that for every $z \in Q_1$,
\[
\begin{split}
\frac{\omega\big(E\big)}{\omega(Q_1(z))} & = \frac{\omega\big(E\big)}{\omega(Q_2)} \frac{\omega(Q_2)}{ \omega(Q_1(z))}  \leq   [\omega]_{A_q}  \frac{\omega\big(E\big)}{\omega(Q_2)} \left(\frac{|Q_2|}{|Q_1(z)|}\right)^q = C([\omega]_{A_q}, n, q)   \frac{\omega\big(E\big)}{\omega(Q_2)}.
\end{split}
\]
Then, using Lemma \ref{doubling} again, we can find $\beta = \beta ([\omega]_{A_q}, n) >0$ such that
\[
\frac{\omega\big(E\big)}{\omega(Q_1(z))} \leq C([\omega]_{A_q}, q, n) \left(\frac{|E|}{|Q_2|} \right) ^{\beta}.
\]
On the other hand, by the weak type (1,1) estimate, Lemma \ref{Hardy-Max-p-p}, we see that
\[
\begin{split}
\big|E\big| & = \big|\big\{Q_1: \M_{Q_2} (|\nabla u |^2) > N \kappa^2 \big\}\big|   \leq \frac{C(n)}{N \kappa^2} \int_{Q_2} |\nabla u|^2 dxdt = \frac{C |Q_2|}{\kappa^2} \norm{|\nabla u|}_{L^2(Q_2)}^2. 
\end{split}
\]
Hence, combining the last two estimates, we can see
\[
\begin{split}
\frac{\omega\big(E\big)}{\omega(Q_1(z))} & \leq C_*([\omega]_q, q, n)  \left( \frac{\norm{\nabla u}_{L^2(Q_2)}}{\kappa }  \right)^{2\beta}, \quad \forall \ z \in Q_1.
\end{split}
\]
Then, by taking 
\begin{equation} \label{kappa-def}
\kappa = \norm{\nabla u}_{L^2(Q_2)} \left( C_*/\e \right) ^{1/(2\beta)},
\end{equation} 
we then obtain
\[
\omega(E) = \omega\big(\big\{Q_1: \M_{Q_2} (|\nabla u'|^2) > N  \big\}\big) \leq \e \omega(Q_1(z)), \quad \forall \ z \in Q_1.
\]
This means that \eqref{add-asump} holds for $u'$. Hence, it follows from \eqref{est-main-1} that
\[
\norm{\nabla u'}_{L^p(Q_1, \omega)} \leq C \left[ \norm{F'}_{L^p(Q_2,\omega)} +  \norm{\M_{\alpha, Q_2}(B')^{2/s}}_{L^{p/2}(Q_2, \omega)}^{1/2} + \omega(Q_1)^{1/p} \right].
\]
This and \eqref{kappa-def} imply that
\[
\begin{split}
\norm{\nabla u}_{L^p(Q_1, \omega)}  & \leq C \left[ \norm{F}_{L^p(Q_2,\omega)} +  \norm{\M_{\alpha, Q_2}(B)^{2/s}}_{L^{p/2}(Q_1, \omega)}^{1/2} + \omega(Q_1)^{1/p} \norm{\nabla u}_{L^2(Q_2)}  \right] \\
& = C \left[ \norm{F}_{L^p(Q_2,\omega)} +  [u]_{s', \lambda, Q_1} \norm{\M_{\alpha, Q_2}(|b|^s)^{2/s}}_{L^{p/2}(Q_2, \omega)}^{1/2} + \omega(Q_1)^{1/p} \norm{\nabla u}_{L^2(Q_2)}  \right].
\end{split}
\]
The proof is then complete.
\ \\
\textbf{Acknowledgement.} This work is partly supported by 
the Simons Foundation, grant number \#~354889.  

\end{document}